\documentclass[11pt, reqno]{amsart}
\usepackage{amsmath, enumerate}
\usepackage{amsfonts}
\usepackage{amssymb}
\usepackage{amsthm}
\usepackage{graphicx}
\usepackage{soul}
\usepackage[usenames, dvipsnames]{color} 
\usepackage{hyperref}
\newtheorem{Theorem}{Theorem}[section]
\newtheorem{Corollary}[Theorem]{Corollary}
\newtheorem{Lemma}[Theorem]{Lemma}
\newtheorem{Proposition}[Theorem]{Proposition}

\theoremstyle{definition}
\newtheorem{Definition}[Theorem]{Definition}

\theoremstyle{remark}

\def \O {\Omega}
\def \bO {\partial \Omega}
\def \R {\mathbb{R}}
\def \Par {\partial_P \mathcal{C}}
\def \e {\epsilon}

\def \a {\alpha}
\def \b {\beta}
\def \t {\tau}
\def \la {\left\langle }
\def \ra {\right\rangle}

\def \det{\mathrm{det}}

\def \ra{\right\rangle}
\def \la{\left\langle}
\def \Id{\mathrm{Id}}

\def\Dbar{\overline D}
\def\xbar{\overline x}
\def\II{\mathrm{II}}
\newcommand{\coord}[2]{\left[#1\right]_{#2}}
\newcommand{\norm}[1]{\lvert #1\rvert}
\def\Omegabar{\bar{\Omega}}
\newcommand{\inner}[2]{\langle #1, #2\rangle}
\def\hstar{h^*}
\def\upperbound{\mathcal{K}}
\def\G{G}
\DeclareMathOperator{\Hess}{Hess}
\def \L{\Delta_{\phi}}
\def \Ric{\text{Ric}_{\phi}}
\DeclareMathOperator{\Vol}{Vol}
\begin{document}

\title[Exponential Convergence of Parabolic Optimal Transport]{Exponential Convergence of Parabolic Optimal Transport on Bounded Domains}
\author[F. Abedin and J. Kitagawa]{Farhan Abedin and Jun Kitagawa$^*$} 
\address{Department of Mathematics, Michigan State University, East Lansing, MI 48824}
\email{abedinf1@msu.edu}
\address{Department of Mathematics, Michigan State University, East Lansing, MI 48824}
\email{kitagawa@math.msu.edu}
\subjclass[2010]{35K96, 58J35}
\thanks{$^*$JK's research was supported in part by National Science Foundation grant DMS-1700094. \\
\hphantom{11l}$^*$ Corresponding author: kitagawa@math.msu.edu}
\begin{abstract}
We study the asymptotic behavior of solutions to the second boundary value problem for a parabolic PDE of Monge-Amp\`ere type arising from optimal mass transport. Our main result is an exponential rate of convergence for solutions of this evolution equation to the stationary solution of the optimal transport problem. We derive a differential Harnack inequality for a special class of functions that solve the linearized problem. Using this Harnack inequality and certain techniques specific to mass transport, we control the oscillation in time of solutions to the parabolic equation, and obtain exponential convergence. Additionally, in the course of the proof, we present a connection with the pseudo-Riemannian framework introduced by Kim and McCann in the context of optimal transport, which is interesting in its own right.
\end{abstract}

\maketitle

\section{Introduction}

Given two smooth domains $\O$, $\O^* \subset \R^n$, two probability measures $\mu$, $\eta$ defined respectively on $\O$ and $\O^*$, and a Borel measurable \emph{cost function} $c: \overline{\O} \times \overline{\O^*} \rightarrow \R$, the optimal transport problem is to find a $\mu$-measurable map $T : \O \rightarrow \O^*$ satisfying $T_{\#}\mu = \eta$ (where $T_{\#}\mu(E) := \mu(T^{-1}(E))$ for all measurable $E \subset \O^*$) such that
\begin{equation}\label{OT}
\int_{\O} c(x, T(x)) \ d\mu(x) = \max_{S_{\#}\mu = \eta} \int_{\O} c(x, S(x)) \ d\mu(x).
\end{equation}
Under mild assumptions on the cost function and the measures, it can be shown that the solution $T$ to \eqref{OT} exists (see, for example, \cite{Brenier91, GangboMcCann96}). If the measures $\mu$ and $\eta$ are absolutely continuous with respect to Lebesgue measure, and $c$ satisfies the bi-twist condition \eqref{bi-twist} below, the map $T$ is $\mu$-a.e. single valued and can be determined by the implicit relation
\begin{align*}
 \nabla_xc(x, T(x))=\nabla u(x),
\end{align*}
where the scalar-valued potential $u$ is a \emph{$c$-convex function} (see Definition \ref{def: c-convex function}) satisfying the Monge-Amp\`ere type equation
\begin{equation}\label{staticPDE}
\begin{cases}
\det[D^2u(x) - A(x, \nabla u(x))] = B(x,\nabla u(x)), \qquad x \in \O\\
T(\O) = \O^*,
\end{cases}
\end{equation}
where $A$ is a matrix-valued function and $B$ is scalar-valued, defined in terms of the cost function $c$ and the densities of the measures $\mu$, $\eta$. The issue of existence and regularity of solutions to the PDE \eqref{staticPDE} has been an active area of research for many years. For higher order regularity results, we refer the reader to \cite{MaTrudingerWang05, TrudingerWang09, Urbas97}.

One possible approach to finding a solution to the PDE above is to solve the parabolic PDE 
\begin{equation}\label{evolPDE}
\begin{cases}
\partial_t u(x,t) = \log \det[D^2u(x,t) - A(x,\nabla u(x,t))] - \log B(x,\nabla u(x,t)), \qquad x \in \O, \ t > 0 \\
\G(x,\nabla u(x,t)) = 0, \qquad x \in \partial \O \\
u(x,0) = u_0(x), \qquad x \in \O
\end{cases}
\end{equation}
for appropriate initial and boundary conditions $u_0$ and $\G$ (see Section \ref{sec: preliminaries}), and view a stationary solution as $t\to\infty$ as a solution to \eqref{staticPDE}. The study of existence, regularity, and asymptotic behavior of solutions to the parabolic problem \eqref{evolPDE} was initiated only recently through the works \cite{Kitagawa12} and \cite{KimStreetsWarren12}. 

The primary contribution of this paper is the following theorem on an exponential convergence rate of solutions to the parabolic equation \eqref{evolPDE}. The notation $C^{k_1}_xC^{k_2}_t$ will denote functions on a space-time domain which are $C^{k_1}$ in the space variable and $C^{k_2}$ in the time variable, with corresponding norms finite. Our main result is as follows:
\begin{Theorem}\label{thm: main}
 Suppose $u\in C^4_xC^3_t(\overline{\O} \times [0, \infty))$ is a solution on $\overline{\O} \times [0, \infty)$ to the parabolic equation \eqref{evolPDE} converging uniformly on $\overline{\O}$ to a stationary solution $u^\infty$ as $t\to\infty$, and $\upperbound$ is a constant such that
\begin{equation}\label{definitionofK}
 \lVert u\rVert_{C^4_xC^2_t(\overline{\O}\times [0, \infty))}+\lVert c\rVert_{C^4(\overline\O \times\overline{\O^*})}\leq \upperbound.
\end{equation}
If the cost function $c$ satisfies the bi-twist condition \eqref{bi-twist}, and $\Omega$ and $\Omega^*$ satisfy the $c$-convexity conditions \eqref{cconvexity} and \eqref{cstarconvexity}, then 
\begin{align*}
 \lVert u(\cdot, t)-u^\infty\rVert_{L^\infty(\Omega)}\leq C_1e^{-C_2t},\quad\forall t\geq 0
\end{align*}
for some constants $C_1$, $C_2>0$ depending only on $\upperbound$ and the dimension $n$.
\end{Theorem}

Previous work of the second author in \cite{Kitagawa12} establishes the existence of a function $u\in C^2_xC^1_t(\overline{\O} \times [0, \infty))$  that solves \eqref{evolPDE} for all times $t \geq 0$ and converges in $C^2(\overline\O)$ to a function $u^{\infty}(\cdot)$ as $t \rightarrow \infty$, where $u^{\infty}(\cdot)$ satisfies the elliptic optimal transport equation \eqref{staticPDE}. Using this result and a bootstrapping argument, we obtain the following corollary.
\begin{Corollary}\label{cor: exp conv for MTW}
 Suppose the cost function $c$ satisfies the bi-twist condition \eqref{bi-twist} and the Ma-Trudinger-Wang condition \eqref{A3}, and that $\Omega$ and $\Omega^*$ satisfy the $c$-convexity conditions \eqref{cconvexity} and \eqref{cstarconvexity} with $\delta$, $\delta^*>0$. Suppose the source and target measures $\mu$ and $\eta$ are absolutely continuous with smooth densities that are bounded away from zero and infinity on $\overline \O$ and $\overline{\O^*}$ respectively. Finally, suppose the initial condition $u_0 \in C^{4,\a}(\Omega)$ for some $\a \in (0,1]$, is locally, uniformly $c$-convex (as in Definition \ref{def: c-convex function}) and satisfies the boundary compatibility conditions \eqref{boundarycompatbility}. Then $u$ satisfies the hypotheses of Theorem \ref{thm: main} above.
\end{Corollary}
\emph{Proof of Corollary:} Under the Ma-Trudinger-Wang condition \eqref{A3} and the uniform $c$- and $c^*$-convexity of the domains (i.e. \eqref{cconvexity} and \eqref{cstarconvexity} with $\delta$, $\delta^*>0$), global $C_x^{2,\a}C^{1,\a}_t$ estimates of the solution $u(x,t)$ to \eqref{evolPDE} were obtained in \cite[Theorems 10.1 and 11.2, and Section 12]{Kitagawa12}. Thus, by applying boundary Schauder estimates for linear uniformly parabolic equations in non-divergence form with uniformly oblique boundary conditions (cf. \cite[Theorem 4.23 and Theorem 4.31]{Lieberman96}) to the linearized equation \eqref{linear}, we obtain the desired higher regularity of $u$. \qedsymbol

\bigskip

The parabolic flow \eqref{evolPDE} on Riemannian manifolds with no boundary was considered by Kim, Streets, and Warren in \cite{KimStreetsWarren12}, under a strong form of the Ma-Trudinger-Wang condition \eqref{A3}. There, the authors prove exponential convergence of the solution $u$ of \eqref{evolPDE} to the solution $u^{\infty}$ of the elliptic equation \eqref{staticPDE} (cf. \cite[Theorem 1.1]{KimStreetsWarren12}). Their proof relies on establishing a Li-Yau type Harnack inequality for solutions to the linearization of $\eqref{evolPDE}$, coupled with the observation that this linearization is actually a heat equation where the elliptic part is a conformal factor times the Laplace-Beltrami operator of a conformal change of a metric defined from the solution of the parabolic evolution itself (cf. \cite[Proposition 5.1]{KimStreetsWarren12} and discussion preceding Proposition \ref{prop: interior expression} below). However, in our case, because there is a nonempty boundary, it is not clear how to prove a full Harnack inequality: the major obstruction is that the Riemannian metric defining the Laplace-Beltrami operator is evolving in time. Although the relevant boundary condition turns out to be the Neumann condition with respect to this time varying metric, it seems that a Harnack inequality may not be expected without additional curvature conditions on the evolution of this metric. In manifolds with nonempty boundary, existing results generally require that the metric itself evolves under some curvature flow such as Ricci flow \cite{BailesteanuCaoPulemotov10} or Gauss curvature flow \cite{Chow91}. While there is a sizeable body of work on differential Harnack inequalities, none of them are directly applicable to the problem \eqref{evolPDE}. We also mention the result \cite{SchnurerSmoczyk03}, which treats a nonlinear evolution equation arising from Gauss curvature flow that resembles \eqref{evolPDE} in the case where the cost function is $c(x, y)=\la x, y\ra$, and with nonempty boundary. The authors of  \cite{SchnurerSmoczyk03} also obtain an exponential convergence result, but assume certain structural assumptions on the function $B$ in \eqref{evolPDE} that are not satisfied in the optimal transport case, and impose additional constraints on the initial data $u_0$.

It turns out that it is not necessary to establish a full Harnack inequality to prove the exponential convergence result. Our approach is to first obtain a one-sided Harnack inequality only for certain special solutions of the linearized equation arising directly from the solution of the parabolic flow. To do this, we prove a sublinearity bound in time for such solutions; in the interior, this can be shown in a manner similar to that of \cite{KimStreetsWarren12} with no boundary, but obtaining the necessary estimates at the boundary will require new techniques. We then make use of the fact that solutions of the parabolic flow give rise to mappings between probability measures that preserve total mass (see Lemma \ref{lem: massbalance} below) to obtain oscillation control in time from this one-sided Harnack inequality for the special solutions.

Let us comment here that one particular motivation for this exponential convergence result comes from numerics for optimal transport. Since the stationary state of \eqref{evolPDE} gives rise to the solution of the optimal transport problem between the measures $\mu$ and $\eta$, one could attempt to implement an algorithm that is initiated with some $c$-convex potential function and flows toward the desired solution via the equation \eqref{evolPDE}. Establishing quantitative rates of convergence for such an algorithm is consequently of paramount importance. One difficulty that should be noted here is that in the case with nonempty boundary, due to compatibility requirements with the boundary condition, there are some restrictions on what can be taken as an initial condition (compare to the case of no boundary, where one can simply take a constant function), and it is not always clear how to generate initial data that will still provide global existence. We plan to explore this issue of finding appropriate initial conditions in future work.

Lastly, we mention that the analysis of the boundary behavior can be carried out from a purely geometric standpoint. In this framework, we find a curious connection with the pseudo-Riemannian metric introduced by Kim and McCann for optimal transport in \cite{KimMcCann10}. We present this analysis in two different ways, one using more traditional ``PDE'' techniques, and another exploiting this (pseudo)-Riemannian framework.

The outline of the remainder of the paper is as follows. In Section \ref{sec: preliminaries} we give the necessary background for the optimal transport problem, including a recap of some previous results for  the parabolic version, and recall the idea behind the proof of exponential convergence on manifolds with no boundary. In Section \ref{sec: boundary sublinear} we handle the boundary case for the sublinear estimates for the linearized operator. For the benefit of the reader, we divide the proof of these estimates into the inner product case and the general cost function case. In Section \ref{sec: exponential} we finally obtain the exponential convergence result from the sublinear estimates derived in the previous sections. The final Section \ref{sec: kim mccann} provides the aforementioned alternative, geometric approach to the boundary sublinearity proof presented in Section \ref{sec: boundary sublinear}.

\section{Preliminaries}\label{sec: preliminaries}
\subsection{Basic Notions from Optimal Transport}
The notations $D^2$, $\nabla$, and $D_\beta$ will be used respectively for the Hessian matrix, the gradient vector, and the directional derivative in direction $\beta$ of a given function with respect to the space variable $x$. Spatial partial derivatives will be denoted by subscript indices, with the actual variable specified when necessary, while $D_x$ and $D_p$ will be used for the derivative matrix of a mapping with respect to the variable in the subscript. We will also follow the convention of summing over repeated indices. Time derivatives will be denoted by $\partial_t$.

When considering a Riemannian manifold $(M,g)$, we will denote the inner product and norm with respect to the metric $g$ by $\la \cdot, \cdot \ra_g$ and $|\cdot|_g$ respectively. The notation $\nabla^g$, $\Hess_g$, $\Delta_g$, and $\text{Ric}_g$ will be used for the gradient, Hessian, Laplacian, and Ricci tensor with respect to $g$. 

Regarding the cost function $c(x,y)$, derivatives in the $x$ variable will be denoted by subscripts preceding a comma, while derivatives in the $y$ variable will be denoted by subscripts following a comma. The notation $c^{i,j}$ denotes the entries of the inverse of the matrix $c_{i,j}$.

 We will assume from here onward that $\O$, $\O^*$ are open, smooth, bounded domains in $\R^n$. The outward-pointing unit normal to $\partial \O$ (resp. $\partial \O^*$) will be denoted by $\nu$ (resp. $\nu^*$). The function $\hstar$ will be a normalized defining function for $\O^*$; i.e. $\hstar = 0$ on $\partial \O$, $\hstar < 0$ on $\O$, and $\nabla \hstar = \nu^*$ on $\partial \O^*$. The measures $\mu$, $\eta$  are assumed to be absolutely continuous with respect to $n$-dimensional Lebesgue measure, with densities $\rho$, $\rho^*$  respectively satisfying the bounds $0 < \lambda \leq \rho, \rho^* \leq \Lambda < \infty$ and the mass balance condition
\begin{equation}\label{massbalance}
\int_{\O} \rho = \int_{\O^*} \rho^*.
\end{equation}
We will also assume $c \in C^{4,\a}(\overline{\O} \times \overline{\O^*})$ for some $\a \in (0,1]$, and that
\begin{align}
y &\mapsto \nabla_x c(x,y) \text{ is a diffeomorphism }\forall x\in \overline\O,\notag\\
x &\mapsto \nabla_y c(x,y) \text{ is a diffeomorphism }\forall y\in \overline{\O^*}.\label{bi-twist}
\end{align}
For any $p \in \nabla_x c(x,\O^*)$ and $x \in \O$, (resp. $q \in \nabla_y c(\O, y)$ and $y \in \O^*$), we denote by $Y(x,p)$ (resp. $X(q,y)$) the unique element of $\O^*$ (resp. $\O$) such that
\begin{equation}\label{A1}
(\nabla_x c)(x,Y(x,p))= p, \quad (\nabla_y c)(X(q,y),y) = q.
\end{equation}
We say \emph{$\O$ is $c$-convex with respect to $\O^*$} if the set $\nabla_y c(\O, y)$ is a convex set for each $y \in \O^*$. Similarly, \emph{$\O^*$ is $c^*$-convex with respect to $\O$} if the set $\nabla_x c(x, \O^*)$ is a convex set for each $x \in \O$. Analytically, these conditions are satisfied if we have
\begin{equation}\label{cconvexity}
\left[\nu^j_i(x) - c^{\ell, k} c_{ij,\ell}(x,y) \nu^k(x) \right] \t^i \t^j \geq \delta \norm{\tau}^2 \qquad \forall \ x \in \bO, \ y \in \overline{\O^*}, \ \tau \in T_x(\bO)
\end{equation}
and
\begin{equation}\label{cstarconvexity}
\left[(\nu^*)^j_i(y) - c^{k, \ell} c_{\ell, ij}(x,y) (\nu^*)^k(x) \right] (\t^*)^i (\t^*)^j \geq \delta^*\norm{\t^*}^2 \qquad \forall \ y \in \bO^*, \ x \in \overline{\O}, \ \tau^* \in T_y(\bO^*)
\end{equation}
for some constants $\delta$, $\delta^*\geq 0$ respectively, where we will always sum over repeated indices. If $\delta$ ($\delta^*$) is strictly positive, we say that \emph{$\O$ is uniformly $c$-convex with respect to $\O^*$} (\emph{$\O^*$ is uniformly $c^*$-convex with respect to $\O$}).

Define the matrix valued function $A$ by $A(x,p):=(D^2_x c)(x, Y(x,p))$. Since $Y(x,p)$ satisfies the equation $(\nabla_x c)(x,Y(x,p)) = p$, we can differentiate implicitly in $p$ to get 
$$(D^2_{x,y}c)(x, Y(x,p)) D_p Y(x,p) = \mathbb{I}_n.$$
Similarly, differentiating the equation $(\nabla_x c)(x,Y(x,p)) = p$ in $x$ gives 
$$(D^2_x c)(x, Y(x,p)) + (D^2_{x,y} c)(x,Y(x,p)) D_x Y(x,p) = 0.$$
We have chosen the convention $(DY)_{\ell m} = Y^{\ell}_m$ for differentiation either in the $x$ or $p$ variables. It follows that
$$A(x,p) = (D^2_x c)(x, Y(x,p)) = - (D_pY)^{-1}(x,p) D_x Y(x,p).$$

\begin{Definition}\label{def: c-convex function}
 A function $\varphi: \Omega\to \R$ is said to be \emph{$c$-convex} if for any point $x_0\in \Omega$, there exists a $y_0\in \Omega^*$ and $\lambda_0\in \R$ such that 
\begin{align*}
 \varphi(x_0)&=c(x_0, y_0)+\lambda_0,\\
 \varphi(x)&\geq c(x, y_0)+\lambda_0,\quad\forall x\in \Omega.
\end{align*}
A function $\varphi \in C^2(\Omega)$ is said to be \emph{locally, uniformly $c$-convex} if $D^2\varphi(x) - A(x, \nabla \varphi(x)) > 0$ as a matrix for every $x \in \overline{\O}$. 
\end{Definition}
Although we will not use it explicitly in this paper, we also mention the, by now well-known, Ma-Trudinger-Wang condition. This condition (or rather a stronger version of it) was first used to obtain interior $C^{2, \alpha}$ regularity of solutions to the elliptic optimal transport equation \eqref{staticPDE} in \cite{MaTrudingerWang05}. It was proven to be a necessary condition for regularity theory in \cite{Loeper09}, and it was shown that classical solutions for the parabolic equation \eqref{evolPDE} exist under the same condition by the second author in \cite{Kitagawa12}.
\begin{Definition}\label{def: MTW} The cost function $c(x,y)$ satisfies the Ma-Trudinger-Wang (MTW) condition if 
\begin{equation}\label{A3}
D_{p_i p_j} A_{k \ell}(x,p)\xi^i \xi^j \eta^k \eta^{\ell} \geq 0 \qquad \text{for all } x \in \overline{\O},\ p \in \nabla_x c(x, \O^*), \ \xi \perp \eta.
\end{equation}
\end{Definition}

\subsection{The Parabolic Optimal Transport Problem}

For a function  $u\in C^4_xC^2_t(\overline\O \times [0, \infty))$  (which, in the sequel, will be the solution to the parabolic optimal transportation problem), we will employ the following notation:
\begin{enumerate}[(i)]
\item $T(x,t) = Y(x,\nabla u(x,t))$
\item $B(x,p) = |\det \ (D^2_{x,y} c)(x,Y(x,p))|\cdot \frac{\rho(x)}{\rho^*(Y(x,p))}$
\item $\G(x,p) = \hstar(Y(x,p))$
\item $\beta(x,t) = \nabla_p \G(x,p) \bigg|_{p = \nabla u}$ 
\item $W(x,t) = D^2u(x,t) - A(x, \nabla u(x,t))$
\end{enumerate}

Using the above notation, we can now precisely state the parabolic optimal transportation problem. We seek to find a function $u\in C^4_xC^2_t(\overline\O \times [0, \infty))$ satisfying the evolution equation
\begin{equation}\label{PDE}
\begin{cases}
\partial_t u(x,t) = \log \det[D^2u(x,t) - A(x,\nabla u(x,t))] - \log B(x,\nabla u(x,t)), \qquad x \in \O, \ t > 0 \\
\G(x,\nabla u(x,t)) = 0, \qquad x \in \partial \O, \ t > 0 \\
u(x,0) = u_0(x), \qquad x \in \O.
\end{cases}
\end{equation}
We require the function $u_0 \in C^{4,\a}(\Omega)$ for some $\a \in (0,1]$, to be locally, uniformly $c$-convex as in Definition \ref{def: c-convex function}, and satisfy
\begin{equation}\label{boundarycompatbility}
\begin{cases}
\hstar(Y(x,\nabla u_0(x))) = 0 \text{ on } \partial \Omega \\
T_0(\O) = \O^*
\end{cases}
\end{equation}
where $T_0(x) := Y(x, \nabla u_0(x))$.

Let us establish some basic facts which will be needed throughout. 

\begin{Lemma} \label{lem: massbalance}
The function $\theta(x,t) := \partial_t u(x,t)$ satisfies
\begin{equation}\label{massbalanceintime}
\int_{\O} e^{\theta (x,t)} \rho(x) \ dx = \int_{\O^*} \rho^*(y) \ dy \qquad \text{for all } t \geq 0.
\end{equation}
\end{Lemma} 
\begin{proof} Differentiating the identity $T(x,t) = Y(x,\nabla u(x,t))$, we obtain
$$T^k_{x_{\ell}}(x,t) = Y^k_{x_{\ell}}(x, \nabla u(x,t)) + Y^k_{p_j}(x,\nabla u(x,t)) u_{x_j x_{\ell}}, \qquad k, \ell = 1, \ldots, n.$$
In matrix notation,
\begin{align}
D_xT(x,t) & = D_xY(x,\nabla u(x,t)) + D_p Y(x,\nabla u(x,t)) D^2u(x,t) \notag\\
& = D_p Y(x,\nabla u(x,t)) (D^2u(x,t) - A(x,\nabla u(x,t)) \notag\\
& = D_p Y(x,\nabla u(x,t)) W(x,t) \notag\\
& = (D^2_{x,y}c)^{-1}(x, Y(x,\nabla u(x,t))) W(x,t).\label{derivative of T}
\end{align}
Consequently,
\begin{equation}\label{detT1}
|\det D_xT(x,t)| = \frac{\det W(x,t)}{|\det (D^2_{x,y} c)(x,T(x,t))|}.
\end{equation}
From \eqref{PDE}, it follows that
\begin{equation}\label{detT2}
e^{\partial_t u(x,t)} \rho(x) = |\det D_xT(x,t)| \rho^*(x,T(x,t)).
\end{equation}
Integrating over $\O$ and using the change of variables formula yields the desired identity. 
\end{proof}

Observe that, by \eqref{massbalanceintime} and the mass balance condition \eqref{massbalance}, $\theta$ must satisfy
\begin{equation}\label{supandinf}
\sup_{\O} \theta(\cdot, t) \geq 0 \quad \text{ and } \quad \inf_{\O} \theta(\cdot, t) \leq 0 \quad \text{ for all } t \geq 0.
\end{equation}

\begin{Lemma}\label{Wbetaisnormal} Let $\nu$ denote the outward pointing unit normal to $\O$, and let $W$ and $\beta$ be defined as above. Then
$$\nu(x) = \frac{W(x,t)\beta(x,t)}{|W(x,t)\beta(x,t)|} \quad \text{for all } (x,t) \in \partial \O \times [0, \infty).$$
\end{Lemma}
\begin{proof} Fix $t \geq 0$. The boundary condition $\G(x,\nabla u(x,t)) = 0$ on $\partial \O$ is equivalent to saying $\hstar(T(x,t)) = 0$ on $\partial \O$. Therefore, by differentiating in any direction $\tau$ tangential to $\partial \O$, we get
$$\hstar_k(T(x,t)) T^k_{x_i}(x,t) \tau^i = 0.$$
In matrix notation,
$$\la W(x,t) (D_pY)^T(x, \nabla u(x,t)) \nabla \hstar(T(x,t)), \t \ra = 0.$$
By definition,
$$\beta(x,t) = (D_p Y)^T(x,\nabla u(x,t)) \nabla \hstar(Y(x,\nabla u(x,t))).$$
Therefore,
$$\la W(x,t) \beta(x,t), \t \ra = 0.$$
It follows that $W\beta$ is parallel to the unit outward pointing normal vector field $\nu$ on $\partial \O$. Since $\hstar < 0$ on $\O$, we can write $W\beta = \chi \nu$, where $\chi \geq 0$. Notice that by \eqref{detT1} and \eqref{detT2}, $W$ is positive definite. By bi-twist \eqref{bi-twist}, and the fact that $\nabla h^* = \nu^*$, we also know $\beta$ is non-zero. Consequently, $\chi = |W\beta|$ is non-zero.
\end{proof} 

\subsection{The Linearized Equation}

Differentiating \eqref{PDE} in $t$ gives the following linear equation for $\theta$:
\begin{equation}\label{linear}
\begin{cases}
\mathcal{L}\theta := w^{ij}\left(\theta_{ij} - D_{p_k} A_{ij} \theta_k \right) + D_{p_k} (\log B) \theta_k - \partial_t \theta = 0 \qquad \text{on } \mathcal{C}_T := \O \times [0,T]\\
D_{\beta} \theta = 0 \qquad \text{on } \partial \O \times [0,T),
\end{cases}
\end{equation}
where $D_{\beta}\theta := \beta \cdot \nabla \theta$, and where in the coefficients, $p = \nabla u(x,t)$. By the global $C^2$ estimates established in \cite{Kitagawa12}, the operator $\mathcal{L}$ is uniformly parabolic and, by \cite[Theorem 7.1 and Theorem 9.2]{Kitagawa12}, the boundary condition $D_{\b} \theta = 0$ is uniformly oblique for all time. Hence, there exist positive constants $c_1, c_2 > 0$ depending only on $\O, \O^*, B, c$ and $u_0$, but independent of $t$, such that $w^{ij} \xi_i \xi_j \geq c_1|\xi|^2$ for all $(x,t) \in \O$ and $\xi \in \R^n$, and $\beta \cdot \nu \geq c_2 > 0$ for all $x \in \partial \O$, $t > 0$.

Solutions to the linearized equation \eqref{linear} satisfy the following maximum principle (see also \cite[Theorem 8.1]{Kitagawa12}).

\begin{Proposition}\label{maxprinciple} Suppose $v$ is a solution to the linearized equation \eqref{linear}. Then
$$\max_{(x,t) \in \mathcal{C}_T} v(x,t)= \max_{x \in \Omega} v(x, 0), \qquad \min_{(x,t) \in \mathcal{C}_T} v(x,t)= \min_{x \in \Omega} v(x, 0).$$
\end{Proposition}

\begin{proof} By the parabolic maximum principle, the maximum of $v$ occurs on the parabolic boundary $\Par_T := (\overline{\Omega} \times \left\{0 \right\}) \cup (\partial \Omega \times (0,T))$. Suppose there exists $(x_0, t_0) \in \partial \Omega \times (0,T)$ such that $v(x_0,t_0) = \max_{(x,t) \in \mathcal{C}_T} v(x,t)$. It then follows from Hopf's Lemma (cf. \cite[Lemma 2.8 and following paragraph]{Lieberman96}) that $D_{\beta}v(x_0,t_0) > 0$. However, this violates the boundary condition $D_{\beta} v = 0$, and so the maximum cannot occur on $\partial \Omega \times (0,T)$. The argument for the minimum follows in similar fashion.
\end{proof}

\subsection{Exponential Convergence on Manifolds with no Boundary}\label{sec: li yau}

The authors of \cite{KimStreetsWarren12} consider the parabolic flow \eqref{PDE} on a Riemannian manifold with no boundary and show exponential convergence of the solution $u$ of \eqref{PDE} to the solution $u^{\infty}$ of the elliptic equation \eqref{staticPDE}. A key ingredient in their proof of exponential convergence is a Li-Yau type Harnack inequality for positive solutions $v$ of the linearized equation $\mathcal{L} v = 0$ (cf. \cite[Theorem 5.2]{KimStreetsWarren12}). This strategy is motivated by the observation that the operator $\mathcal{L}$ is a heat-type equation with respect to the time-varying Riemannian metric $g$ with components $g_{ij} = w_{ij}$ (see discussion preceding Proposition \ref{prop: interior expression} below). 

Let us recall the main ideas behind the exponential convergence result in \cite{KimStreetsWarren12}. The crucial ingredient is a Li-Yau type Harnack inequality satisfied by positive solutions of the linearized equation (cf. \cite[Theorem 5.2]{KimStreetsWarren12}), whose proof is inspired by the seminal work \cite{LiYau86} concerning heat equations on Riemannian manifolds. Suppose $v$ is a positive solution to the linearized equation $\mathcal{L} v = 0$ on $\mathcal{C}_T$, where $T > 0$ is chosen to be sufficiently large. Let $f = \log v$ and consider the quantity
\begin{equation}\label{aux}
F = t(|\nabla^g f|^2_g - \a \partial_t f) = t(w^{ij}f_i f_j - \a \partial_t f),
\end{equation}
where $\a > 0$ is a constant to be determined and $\nabla^g$ denotes the gradient of a function with respect to the metric $g$. It is shown in \cite[Theorem 5.2]{KimStreetsWarren12} that $F$ is sublinear in $t$ everywhere in $\mathcal{C}_T$; that is, there exist constants $C_1, C_2 > 0$ (independent of $T$) such that $F(x,t) \leq C_1 + C_2 t$ for all $x \in \O,\  t \in [0, T]$. The sublinearity in $t$ of $F$ implies the differential Harnack inequality 
\begin{equation}\label{differentialharnack}
w^{ij}f_i f_j - \a \partial_t f \leq \frac{C_1}{t} + C_2,
\end{equation}
for some possibly different constants $C_1$ and $C_2>0$.
A standard argument (see, for instance, \cite[pg. 4345, Proof of Theorem 5.2]{KimStreetsWarren12}) then yields the parabolic Harnack inequality
\begin{equation}\label{harnack}
\sup\limits_{\Omega} v(\cdot, t) \leq C \inf\limits_{\Omega} v\left(\cdot, t+1\right) \qquad \text{for all } t \geq 1,
\end{equation}
where $C > 0$ is a constant independent of $t$.
Once \eqref{harnack} is obtained, a standard ``oscillation-decay-in-time'' argument (cf. \cite[Section 7.1]{KimStreetsWarren12}) shows that $\theta$ converges to a constant function on $\O$ as $t \rightarrow \infty$. Invoking \eqref{supandinf}, we conclude that $\theta \equiv 0$, and so $u(\cdot,t)$ converges as $t \rightarrow \infty$ to a function $u^{\infty}(\cdot)$ solving \eqref{staticPDE}.

We now provide the argument for sublinearity of $F$. The proof relies on an important parabolic inequality satisfied by $F$, \eqref{parabolicinequality},  which we will prove in Proposition \ref{prop: interior expression} below. 

\begin{Proposition}\label{prop: sublinearity if no boundary max}
 If $F$ does not attain a positive maximum on $\partial \Omega \times (0, T)$, then there exist constants $C_1'$ and $C'_2>0$ independent of $T$ such that 
\begin{align}\label{sublinearity inequality}
 F(x, t)\leq C'_1+C'_2t,\quad \forall (x, t)\in \mathcal{C}_T.
\end{align}
\end{Proposition}

\begin{proof} First note that $F(\cdot, 0) \equiv 0$ because $\inf\limits_{\Omega} v(\cdot,0) > 0$, and so the bound holds at $t = 0$. Suppose there exists a first time $\tau \in (0,T)$ such that $F(y, \tau) \geq C'_1 + C'_2 \tau$ for some $y \in \O$. By going further in time if necessary, we may assume there exists a point $(x_0, t_0) \in \overline{\O} \times (0, T]$ such that $F(x_0, t_0) > C'_1 + C'_2t_0$ and $F$ attains a local maximum at $(x_0,t_0)$. If $(x_0,t_0)$ is an interior point of $\mathcal{C}_T$, it follows from \eqref{parabolicinequality} that
$$C_1 F(x_0,t_0)^2 - F(x_0,t_0) - C_2 t_0^2 \leq 0,$$
from which we conclude
\begin{equation}\label{upperboundforF}
F(x_0,t_0) \leq \frac{1 + \sqrt{1 + 4C_1C_2t_0^2}}{2C_1} \leq \tilde{C}_1 + \tilde{C}_2t_0
\end{equation}
for a different set of constants $\tilde{C}_1, \tilde{C}_2  > 0$ and for $t_0 > 0$ sufficiently large. If $C'_1, C'_2$ were chosen at the beginning to satisfy $C'_1 > \tilde{C}_1$ and $C'_2 > \tilde{C}_2$, then we reach a contradiction based on \eqref{upperboundforF}.
\end{proof}
Thus it is clear that on a manifold with no boundary, Proposition \ref{prop: sublinearity if no boundary max} combined with the discussion above yields exponential convergence, as is shown in \cite{KimStreetsWarren12}.

We finish this section by establishing the parabolic inequality \eqref{parabolicinequality} satisfied by $F$. It is shown in \cite[Proposition 5.1]{KimStreetsWarren12} that if $n\geq 3$ and
$$\psi(x,t) := \left(\frac{\rho^*(T(x,t))^2 \ \det D_xT(x,t)}{\norm{\det D^2_{x,y}c(x,T(x,t))}}\right)^{\frac{1}{n-2}},$$
then
$$\mathcal{L}v = \psi \Delta_{\psi g} v - \partial_t v,$$
where $\Delta_{\psi g}$ is the Laplace-Beltrami operator with respect to the time-varying metric $\psi g$ with $g_{ij} :=  w_{ij}$. By adapting the proof of the differential Harnack inequality for the heat equation established in \cite{LiYau86}, the authors of \cite{KimStreetsWarren12} establish a parabolic inequality for $F$ similar to  \eqref{parabolicinequality} in the case of manifolds with no boundary of dimension $n \geq 3$. The case $n=2$ is treated in \cite{KimStreetsWarren12} through the introduction of a third dummy dimension in a manner that causes the solution $u$ of \eqref{PDE} to have a product structure (cf. \cite[Subsection 7.1.2]{KimStreetsWarren12} for details). In the presence of boundary, such an argument for dealing with the two-dimensional case is almost certain to fail due to the requirement of uniform $c$- and $c^*$-convexity of the domains involved.

We elect to take a slightly different approach as introduced in \cite[Section 3]{Warren14}, which considers the weighted Laplacian $\L := \Delta_g - \la \nabla^g \phi, \nabla^g \cdot \ra_g$ for the manifold with density $(\O, g, e^{-\phi} d\Vol_g)$, where
$$\phi(x,t):=\log\left(\frac{\norm{\det D^2_{x,y}c(x,T(x,t))}}{\rho^*(T(x,t))^2 \ \det D_xT(x,t)}\right)^{\frac{1}{2}}.$$
By a simple calculation using coordinates, it is possible to verify that $\mathcal{L} = \L - \partial_t$. This approach has the advantage that the case of dimension $n=2$ does not need to be treated separately.

\begin{Proposition}\label{prop: interior expression}
 Under the same hypotheses as Theorem \ref{thm: main}, there exist constants $C_1$, $C_2$, and $C_3>0$, depending only on the constant $\upperbound$ defined in $\eqref{definitionofK}$ and the dimension $n$, such that whenever $v$ satisfies $\mathcal{L}v =0$,
 \begin{equation}\label{parabolicinequality}
\mathcal{L} F + 2 \la \nabla^g f, \nabla^g F \ra_g \geq \frac{1}{t}\left(C_1 F^2 - F - C_2 t^2 + C_3 t  |\nabla^g f|_g^2F\right).
\end{equation}
\end{Proposition}

\begin{proof}
We recall the well-known weighted Bochner formula
\begin{equation}\label{Bochner}
\L \left(|\nabla^g f|_g^2 \right) = 2||\Hess_g f||^2 + 2\la \nabla^g f, \nabla^g(\L f) \ra_g + 2\Ric(\nabla^g f, \nabla^g f),
\end{equation}
where $\Ric := \text{Ric}_g + \Hess_g \phi$. Clearly, $\Ric \geq -\mathcal{K}$, where $\mathcal{K}$ is defined in \eqref{definitionofK}. Since $\mathcal{L}v =0$, the function $f := \log v$ solves the equation
\begin{equation}\label{eqnforf}
\partial_t f = \L f + |\nabla^g f|_g^2.
\end{equation}
Consider the auxiliary function
$$F:=t\left(|\nabla^g f|_g^2 - \a \partial_t f\right), \qquad \a > 0.$$
By using \eqref{Bochner}, we obtain
\begin{align*}
\L F & = t\left(\L (|\nabla^g f|_g^2) - \a \L (\partial_t f) \right) \\
& = t\left(2||\Hess_g f||^2 + 2\la \nabla^g f, \nabla^g(\L f) \ra_g + 2\Ric(\nabla^g f, \nabla^g f) - \a \L (\partial_t f)\right).
\end{align*}
Direct computation shows that
$$\L (\partial_t f) \leq \partial_t (\L f) + C(||\Hess_g f|| + |\nabla^g f|_g),$$
where $C = C(\partial_t g, \partial_t \nabla g, \partial_t \nabla \phi) \geq 0$ depends only on $\upperbound$. Therefore,
\begin{align*}
\L F & \geq t\left(2||\Hess_g f||^2 + 2\la \nabla^g f, \nabla^g(\L f) \ra_g + 2\Ric(\nabla^g f, \nabla^g f) - \a \partial_t (\L f) -\a C(||\Hess_g f|| + |\nabla^g f|_g) \right) \\
& \geq t\left(||\Hess_g f||^2 + 2\la \nabla^g f, \nabla^g(\L f) \ra_g - \a \partial_t (\L f) -  C_1|\nabla^g f|_g^2 - C_2 \right)
\end{align*}
where we have used Cauchy's inequality and the lower bound for $\Ric$. From \eqref{eqnforf} and the definition of $F$, it follows that
$$\L f = -\left(\frac{F}{t} + (\a - 1) \partial_t f \right).$$
Therefore,
\begin{align*}
2\la \nabla^g f, \nabla^g(\L f) \ra_g & = - 2\la \nabla^g f, \nabla^g \left(\frac{F}{t} + (\a - 1) \partial_t f \right) \ra_g \\
& = -\frac{2}{t}\la \nabla^g f , \nabla^g F \ra_g - 2(\a - 1) \la \nabla^g f, \nabla^g(\partial_t f) \ra_g.
\end{align*}
Furthermore,
$$\partial_t F = \frac{F}{t} + t\left(\partial_t |\nabla^g f|_g^2 - \a \partial^2_t f \right).$$
Therefore,
\begin{align*}
-\a \partial_t (\L f) & = \a \partial_t \left(\frac{F}{t} + (\a - 1) \partial_t f \right) \\
& = \a \left(\frac{\partial_t F}{t} - \frac{F}{t^2} + (\a - 1) \partial^2_t f \right) \\
& = \a \left(\frac{\partial_t F}{t} - \frac{F}{t^2}\right) + (\a - 1) \a \partial^2_t f \\
& = \a \left(\frac{\partial_t F}{t} - \frac{F}{t^2}\right) + (\a - 1) \left(\frac{F}{t^2} - \frac{\partial_t F}{t} + \partial_t|\nabla^g f|_g^2 \right) \\
& = \frac{\partial_t F}{t} - \frac{F}{t^2} + (\a - 1) \partial_t|\nabla^g f|_g^2
\end{align*}
It follows that
\begin{align*}
2\la \nabla^g f, \nabla^g(\L f) \ra_g - \a \partial_t (\L f) & = \frac{1}{t}\left(\partial_t F - 2\la \nabla^g f , \nabla^g F \ra_g - \frac{F}{t} \right) + (\a - 1) \left(\partial_t|\nabla^g f|_g^2 - 2\la \nabla^g f, \nabla^g(\partial_t f) \ra_g \right) \\
& \geq \frac{1}{t}\left(\partial_t F - 2\la \nabla^g f , \nabla^g F \ra_g - \frac{F}{t} \right) - C_3|\nabla^g f|_g^2,
\end{align*}
where we have used the fact
$$\partial_t|\nabla^g f|_g^2 \leq 2\la \nabla^g f, \nabla^g(\partial_t f) \ra_g + \gamma|\nabla^g f|_g^2,$$
for some constant $\gamma = \gamma(\partial_t g) \geq 0$ depending only on $\mathcal{K}$. Inserting the above inequality into the lower bound for $\L F$ yields
$$\L F \geq t\left(||\Hess_g f||^2 + \frac{1}{t}\left(\partial_t F - 2\la \nabla^g f , \nabla^g F \ra_g - \frac{F}{t} \right) - C_4|\nabla^g f|_g^2 - C_2 \right).$$
Now since $\L f = \Delta_g f - \la \nabla^g \phi , \nabla^g f \ra_g$, we have
\begin{align*}
(\Delta_g f)^2 = (\L f + \la \nabla^g \phi , \nabla^g f \ra_g)^2 & = (\L f)^2 + \la \nabla^g \phi , \nabla^g f \ra_g^2 + 2(\L f)\la \nabla^g \phi , \nabla^g f \ra_g \\
& \geq (\L f)^2 + \la \nabla^g \phi , \nabla^g f \ra_g^2 - \frac{(\L f)^2}{2} - 2\la \nabla^g \phi , \nabla^g f \ra_g^2 \\
& = \frac{(\L f)^2}{2} - \la \nabla^g \phi , \nabla^g f \ra_g^2 \\
& \geq \frac{(\L f)^2}{2} - |\nabla^g \phi|_g^2|\nabla^g f|_g^2.
\end{align*}
Therefore, by the arithmetic-geometric mean inequality, we have
$$||\Hess_g f||^2 \geq \frac{1}{n}(\Delta_g f)^2 \geq \frac{(\L f)^2}{2n} - \frac{1}{n}|\nabla^g \phi|_g^2|\nabla^g f|_g^2.$$
Since $|\nabla^g \phi|_g \leq \upperbound$, we obtain
$$\L F \geq \frac{t}{2n}(\L f)^2 + \partial_t F - 2\la \nabla^g f , \nabla^g F \ra_g - \frac{F}{t} - C_5t|\nabla^g f|_g^2 - C_2t.$$
Finally, by \eqref{eqnforf}, and after relabeling constants, we conclude that
$$\L F + 2\la \nabla^g f , \nabla^g F \ra_g - \partial_t F \geq \frac{1}{t}\left[C_1 t^2(|\nabla^g f|_g^2 - \partial_t f)^2 - F - C_2 t^2|\nabla^g f|_g^2 - C_3 t^2 \right].$$

Here the constants $C_1, C_2, C_3 > 0$ depend only on up to fourth order derivatives of the cost function (through $\Hess_g \phi$) and the $C^4_xC^1_t$ norm of the solution $u$ to \eqref{PDE} (through the time derivative of $g$ and bounds on the Ricci curvature of $g$), hence only on $\upperbound$ and on the dimension $n$.

Let $y = |\nabla^g f|_g^2$ and $z = \partial_t f$. Then for any $\a, \e, \delta > 0$, we have the identity
$$(y-z)^2 = \left(\frac{1}{\a} - \frac{\e}{2} \right) (y - \a z)^2 + \left(1 - \frac{\e}{2} - \delta - \frac{1}{\a} \right)y^2 + \left(1 - \a + \frac{\e}{2}\a^2 \right)z^2 + \e y(y-\a z) + \delta y^2.$$
We now choose $\a, \e > 0$ such that
\begin{itemize}
\item[(i)] $1 - \frac{\e}{2} - \frac{1}{\a} > 0$
\item[(ii)] $1 - \a + \frac{\e}{2}\a^2 \geq 0$
\item[(iii)] $\frac{1}{\a} - \frac{\e}{2} > 0.$
\end{itemize}
Note that these conditions impose the restriction $\a>1$. A direct verification shows that $\a = 2$ and $\e = \frac{1}{2}$ satisfy the above inequalities. We then choose $\delta =\frac{1}{8} \in \left(0, 1 - \frac{\e}{2} - \frac{1}{\a} \right) = (0, \frac{1}{4})$. With these choices of $\a, \e, \delta$, we obtain (discarding the second and third terms in the expansion, and using that $F=t(y-\alpha z)$)
$$\L F + 2\la \nabla^g f , \nabla^g F \ra_g - \partial_t F \geq \frac{1}{t}\left[C_1 t^2 \left\{ \frac{F^2}{4t^2}  + y\frac{F}{2t} + \frac{y^2}{8} \right\} - F - C_2t^2 y - C_3 t^2 \right].$$
Using Cauchy's inequality, we may eliminate the $-C_2t^2 y$ and $\dfrac{C_1t^2y^2}{8}$ terms to get
$$\L F + 2\la \nabla^g f , \nabla^g F \ra_g - \partial_t F \geq \frac{1}{t}\left[C_1 t^2 \left\{\frac{F^2}{4t^2}  +  y\frac{F}{2t}\right\} - F - C_4 t^2 \right].$$
Relabeling constants, we have thus established an inequality of the form \eqref{parabolicinequality}.
\end{proof}

\section{Sublinearity of \emph{F} on Domains with Boundary}\label{sec: boundary sublinear}

On a domain with boundary, one must deal with the possibility that $F$ attains a maximum at a point $(x_0, t_0) \in \Par_T = (\overline{\Omega} \times \left\{0 \right\}) \cup (\partial \Omega \times (0,T))$, the parabolic boundary of the cylinder $\mathcal{C}_T$. Since $F = 0$ on $\overline{\Omega} \times \left\{0 \right\}$, it suffices to assume $(x_0,t_0) \in \partial \Omega \times (0,T)$. The original argument of Li and Yau (cf. \cite[proof of Theorem 1.1]{LiYau86}) in the case of the heat equation eliminates the possibility of $F$ attaining a non-negative maximum on $\partial \Omega \times (0,T)$ by means of a contradiction to Hopf's Lemma. For this, they require two additional hypotheses: namely, the solution to the heat equation also satisfies a Neumann boundary condition, and that the boundary is mean-convex. 

We will obtain a similar contradiction to Hopf's Lemma only for the particular non-negative solution $\Theta(x,t) := \sup\limits_{\O} \theta(\cdot, 0) - \theta(x,t)$ of the linearized equation \eqref{linear} (as well as for translations of $\Theta$ in time) by exploiting the boundary condition $D_{\beta} \Theta = 0$ on $\partial \O \times [0,T]$, and using the assumption that the domains $\O$, $\O^*$ are respectively $c$-convex and $c^*$-convex. This gives the desired sublinearity at the boundary of the corresponding function $F$ defined in \eqref{aux} and establishes the Harnack inequality \eqref{harnack} for $\Theta$, which turns out to be sufficient to prove the exponential convergence of $u(\cdot,t)$ to the steady state solution $u^{\infty}(\cdot)$ as $t \rightarrow \infty$ (cf. Section \ref{sec: exponential}). As mentioned in the introduction, it is unclear if such a sublinearity estimate at the boundary holds for an arbitrary non-negative solution $v$ of the linearized equation \eqref{linear}.

Let us carry on with the proof of the sublinearity of $F$ outlined in Proposition \ref{prop: sublinearity if no boundary max}, now assuming there exists $(x_0, t_0) \in \partial \O \times (0,T)$ such that $F(x_0, t_0) > C'_1 + C'_2t_0$ and that $F$ attains a local maximum at $(x_0, t_0)$. It follows from \eqref{parabolicinequality} that, in a spherical cap near $(x_0,t_0)$, we have 
$$\mathcal{L} F + 2 \la \nabla^g f, \nabla^g F \ra_g \geq 0.$$
By the uniform obliqueness of $\beta$ and Hopf's Lemma, it follows that $D_{\beta} F(x_0,t_0) > 0$. Anticipating a contradiction, we proceed to explicitly compute $D_{\beta}F(x_0,t_0)$. We first make a rotation centered at $x_0$ so the directions $e_1, \ldots, e_{n-1}$ form an orthonormal basis for the tangent space to $\partial \O$ at $x_0$, and the direction $e_n$ is the outward pointing unit normal direction to $\partial \O$ at $x_0$. Differentiating $F$ in these coordinates, we find that
\begin{align*}
D_{\beta}F(x_0,t_0) & = D_{\beta}\bigg|_{(x_0,t_0)} t(w^{ij}f_i f_j - \a \partial_t f) \\
& = t_0\left[\left(D_{\beta} w^{ij}\right) f_i f_j + 2 w^{ij} \left(D_{\beta} f_i \right) f_j - \a D_{\beta}(\partial_t f) \right]\bigg|_{(x_0,t_0)} \\
& = t_0\left[- w^{i\ell} w^{jk} \left( D_{\beta} w_{\ell k} \right) f_i f_j  + 2 w^{ij}  \left(\left(D_{\beta} f\right)_i - \beta^k_i f_k\right) f_j - \a  \left(\partial_t \left(D_{\beta} f\right)- (\partial_t \beta^k) f_k\right) \right]\bigg|_{(x_0,t_0)}.
\end{align*}
Now since $D_{\b}f =\frac{D_\beta v}{v}= 0$ on $\bO$, we have $\partial_t \left(D_{\beta} f\right) = 0$ and $\left(D_{\beta} f\right)_i  = 0$ for $i = 1, \ldots, n-1$. Therefore,
$$D_{\beta}F(x_0,t_0) = t_0\left[- w^{i\ell} w^{jk} \left( D_{\beta} w_{\ell k} \right) f_i f_j  - 2 w^{ij}   \beta^k_i f_k f_j  + 2w^{nj}f_j(D_{\b} f)_n + \a (\partial_t \beta^k) f_k \right]\bigg|_{(x_0,t_0)}.$$
We claim $w^{nj}f_j  = 0$ at $(x_0,t_0)$. By Lemma \ref{Wbetaisnormal}, $W\beta$ is parallel to the outward pointing unit normal vector $\nu$ on $\partial \O$, so $\nu = \frac{1}{\chi} W\b$, where $\chi := |W\b|$. Again since $D_{\b}f = 0$ on $\partial \O$,
$$0 = \la \b, \nabla f \ra = \la W^{-1} W\b, \nabla f \ra = \la W\b, W^{-1} \nabla f \ra.$$
Hence, 
\begin{align}\label{def of tau}
\tau := W^{-1} \nabla f
\end{align}
 is tangent to $\partial \O$. In the coordinate system defined above, we have $\nu(x_0,t_0) = e_n$, and so $\t_n(x_0,t_0) = 0$. Since $\tau^n = w^{nj}f_j$, the claim is proved. 
It follows that
\begin{equation}\label{DbetaFprelim}
D_{\beta}F(x_0,t_0) = t_0\left[-\left( D_{\beta} w_{k \ell} \right) \t^k \t^{\ell}  - 2 \beta^k_i f_k \tau^i + \a (\partial_t \beta^k)  f_k \right]\bigg|_{(x_0,t_0)}.
\end{equation}
Note that since $\t_n = 0$ at $(x_0,t_0)$, it suffices to sum the indices in the first term over $k, \ell = 1, \ldots, n-1$.

\subsection{Inner product cost} We first show how to explicitly compute $D_{\beta}F(x_0,t_0)$ in the case when the cost function is given by the Euclidean inner product on $\R^n$ (which is known to be equivalent to taking the cost function to be the Euclidean distance squared). There are a number of simplifications in this case, as $Y(x, p)=p$, $W(x, t)=D^2u(x, t)$, and $c$- and $c^*$-convexity of sets and functions reduce to the usual notions of convexity of the domains $\O$ and $\O^*$.

\begin{Proposition}\label{DbetaFquadcost}
If $c(x, y)=\la x, y\ra$
\begin{equation}\label{DbetaFforquadcost}
D_{\beta}F(x_0,t_0) = t_0\left[-\chi  \la (D \nu) \tau, \tau\ra - \la D^2\hstar(\nabla u) \nabla f, \nabla f \ra + \a \la D^2 \hstar(\nabla u) \nabla \theta, \nabla f \ra \right]\bigg|_{(x_0,t_0)}.
\end{equation}
\end{Proposition}

\begin{proof}
We have
$$W = D^2 u, \qquad \beta = \nabla \hstar(\nabla u).$$
Consequently, $\nu = \frac{1}{\chi}(D^2u) \b$. Differentiating $\nu^k$ in the $e_{\ell}$ direction for $k, \ell = 1, \ldots, n-1$, we find
\begin{align*}
\nu^k_{\ell} & = \left(\frac{1}{\chi} u_{kr} \b^r \right)_{\ell} \\
& = \frac{1}{\chi}\left(u_{\ell kr} \b^r + u_{kr} \b^r_{\ell} \right) - \frac{\chi_{\ell}}{\chi^2} \left(u_{kr} \b^r \right) \\
& = \frac{1}{\chi}\left(D_{\b}u_{\ell k}+ u_{kr} \b^r_{\ell} \right) - (\log \chi) _{\ell} \nu^k.
\end{align*}
Solving for $D_{\b}u_{\ell k}$, we obtain
$$D_{\b}u_{\ell k} = \chi \nu^k_{\ell} - u_{kr} \b^r_{\ell} + \chi (\log \chi)_{\ell} \nu^k.$$
Therefore at $(x_0,t_0)$, we have (recall \eqref{def of tau})
\begin{align*}
-\left( D_{\beta} u_{\ell k} \right) \tau^{\ell}\tau^{k} & = -\left(\chi \nu^k_{\ell} - u_{kr} \b^r_{\ell} + \chi  (\log \chi)_{\ell} \nu^k \right) \tau^{\ell}\tau^{k} \\
& = -\chi \nu^k_{\ell} \tau^{\ell}\tau^{k} + u_{kr} \tau^{k} \b^r_{\ell}\tau^{\ell} \\
& = -\chi  \nu^k_{\ell} \tau^{\ell}\tau^{k} + f_r \b^r_{\ell}\tau^{\ell}
\end{align*}
where we sum the indices $k, \ell$ from $1$ to $n-1$. Substituting this into \eqref{DbetaFprelim} gives
$$D_{\beta}F(x_0,t_0) = t_0\left[-\chi \nu^k_{\ell} \tau^{\ell}\tau^{k} - \beta^k_i f_k \tau^i + \a (\partial_t \beta^k) f_k \right]\bigg|_{(x_0,t_0)}.$$
Since $\beta(x,t) = \nabla\hstar(\nabla u(x,t))$, we find that
$$\beta^k_i f_k \tau^i = \hstar_{k \ell}(\nabla u) u_{\ell i} f_k \t^i = \hstar_{k \ell}(\nabla u) f_k u_{\ell i} \t^i = \hstar_{k \ell}(\nabla u) f_k f_{\ell},$$
and
$$(\partial_t \beta^k)  f_k = \hstar_{k \ell}(\nabla u) (\partial_t u_{\ell}) f_k = \hstar_{k \ell}(\nabla u) \theta_{\ell} f_k,$$
hence \eqref{DbetaFforquadcost} follows.
\end{proof}

\subsection{General Cost} We now show how to explicitly compute $D_{\beta}F(x_0,t_0)$ in the case of a general cost.

\begin{Proposition}\label{DbetaFexpression}
\begin{equation}\label{DbetaFgencost3}
D_{\beta}F(x_0,t_0) = t_0\left[-\chi \left(\nu^j_i  - c^{r,\ell} c_{ij,r}\nu^{\ell}\right) \tau^{i}\tau^{j} - \G_{p_k p_s}(x,\nabla u) f_k  f_s + \a \ \G_{p_k p_s}(x,\nabla u) f_k  \theta_s \right]\bigg|_{(x_0,t_0)}.
\end{equation}
\end{Proposition}

\begin{proof} We have
$$w_{jk}(x,t) = u_{jk}(x,t) - c_{jk}(x, T(x,t)), \qquad \beta^k(x,t) = \hstar_{\ell}(Y(x,\nabla u(x,t))) Y^{\ell}_{p_k}(x,\nabla u(x,t)).$$
Recall that $\nu = \frac{1}{\chi}W\beta$. As in the case of the inner product cost, we differentiate $\nu^j$ in the $e_i$ direction for $i,j = 1, \ldots, n-1$ to get
\begin{align*}
\nu^j_i & = \left(\frac{1}{\chi} w_{jk} \b^k \right)_i \\
& = \frac{1}{\chi}\left((w_{jk})_i \b^k + w_{jk}\b^k_i \right) - \frac{\chi_{i}}{\chi^2} \left(w_{jk} \b^k \right) \\
& = \frac{1}{\chi}\left((w_{jk})_i \b^k + w_{jk}\b^k_i \right)  - (\log \chi)_i \nu^j.
\end{align*}
Differentiating $w_{jk}$ gives
\begin{align*}
\left(w_{jk} \right)_i & = u_{jki} - c_{jki} - c_{jk,r}T^r_i \\
& = \left(w_{ij} \right)_k+ c_{ij,r}T^r_k - c_{jk,r}T^r_i \\
& = \left(w_{ij} \right)_k + c_{ij,r}c^{r,\ell}w_{\ell k} - c_{jk,r}c^{r,\ell}w_{\ell i}.
\end{align*}
where we have used \eqref{derivative of T} in the final line. Therefore,
\begin{align*}
\nu^j_i & = \frac{1}{\chi}\left((w_{jk})_i \b^k + w_{jk}\b^k_i \right) - (\log \chi)_i \nu^j \\
& = \frac{1}{\chi}\left(\left[\left(w_{ij} \right)_k + c_{ij,r}c^{r,\ell}w_{\ell k} - c_{jk,r}c^{r,\ell}w_{\ell i}  \right] \b^k + w_{jk}\b^k_i \right) - (\log \chi)_i \nu^j\\
& = \frac{1}{\chi}\left(D_{\b}w_{ij}+ \left[c_{ij,r}c^{r,\ell}w_{\ell k} - c_{jk,r}c^{r,\ell}w_{\ell i}  \right] \b^k + w_{jk}\b^k_i \right) -(\log \chi)_i \nu^j.
\end{align*}
Solving for $D_{\b}w_{ij}$, we obtain
$$D_{\b}w_{ij} = \chi \nu^j_i - \left[c_{ij,r}c^{r,\ell}w_{\ell k} - c_{jk,r}c^{r,\ell}w_{\ell i}  \right] \b^k - w_{jk}\b^k_i + \chi (\log \chi)_i \nu^j.$$
Therefore, at $(x_0,t_0)$, we have (again using \eqref{def of tau})
\begin{align*}
-\left( D_{\beta} w_{ij} \right) \tau^i\tau^j & = -\left( \chi \nu^j_i - \left[c_{ij,r}c^{r,\ell}w_{\ell k} - c_{jk,r}c^{r,\ell}w_{\ell i}  \right] \b^k - w_{jk}\b^k_i + \chi \partial_i(\log \chi) \nu^j \right) \t^i \t^j \\
& = -\left(\chi \nu^j_i  - c_{ij,r}c^{r,\ell}w_{\ell k} \b^k + c_{jk,r}c^{r,\ell}w_{\ell i} \b^k \right)\tau^{i}\tau^{j} + w_{jk}\b^k_i \tau^{i}\tau^{j} \\
& = -\chi \left(\nu^j_i - c^{r,\ell} c_{ij,r}\nu^{\ell}\right) \tau^{i}\tau^{j} - c^{r,\ell} c_{jk,r}f_{\ell} \b^k \tau^{j} + f_k \b^k_i \tau^{i} \\
& = -\chi \left(\nu^j_i  - c^{r,\ell} c_{ij,r}\nu^{\ell}\right) \tau^{i}\tau^{j} - c_{jk,r}\hstar_s  Y^s_{p_k} Y^r_{p_{\ell}} f_{\ell} \tau^{j} + f_k \b^k_i \tau^{i}
\end{align*}
where we sum the indices $i,j$ from $1$ to $n-1$. It follows from \eqref{DbetaFprelim} that 
\begin{equation}\label{DbetaFgencost1}
D_{\beta}F(x_0,t_0) = t_0\left[-\chi \left(\nu^j_i - c^{r,\ell} c_{ij,r}\nu^{\ell}\right) \tau^{i}\tau^{j} - c_{jk,r}\hstar_s  Y^s_{p_k} Y^r_{p_{\ell}}  f_{\ell} \tau^{j} - f_k \b^k_i \tau^{i} + \a \partial_t \beta^k f_k \right]\bigg|_{(x_0,t_0)}.
\end{equation}
We compute
\begin{align}\label{derivative of beta}
\beta^k_i  = \hstar_{\ell r}\left(Y^r_{x_i} + Y^r_{p_s} u_{si}\right)Y^{\ell}_{p_k} + \hstar_{\ell} \left(Y^{\ell}_{p_k x_i} + Y^{\ell}_{p_k p_s} u_{si}\right).
\end{align}
To simplify the first term, recall the identity (see \eqref{derivative of T})
$$Y^r_{x_i} + Y^r_{p_s} u_{si} = Y^r_{p_s} w_{si}.$$
For the second term in \eqref{derivative of beta}, we differentiate the equation $c_{i,\ell} Y^{\ell}_{p_k} = \delta_{ik}$ with respect to $p_s$ and $x_i$ to obtain
$$Y^{\ell}_{p_kp_s} = -c^{\ell,j} c_{j,rq} Y^r_{p_k} Y^q_{p_s}$$
and
$$Y^{\ell}_{p_k x_i} =  -c^{\ell,j} c_{ij,r}Y^r_{p_k} + c^{\ell,j} c_{j, rq} Y^r_{p_k} Y^q_{p_s} c_{si} = -c_{ij,r}Y^{\ell}_{p_j} Y^r_{p_k}  - Y^{\ell}_{p_kp_s} c_{si}.$$
Therefore,
$$Y^{\ell}_{p_k x_i} + Y^{\ell}_{p_kp_s} u_{si} =  - c_{ij,r}Y^{\ell}_{p_j} Y^r_{p_k} + w_{si} Y^{\ell}_{p_kp_s}.$$
Substituting these into the expression \eqref{derivative of beta} gives
$$\beta^k_i = \hstar_{\ell r} Y^{\ell}_{p_k} Y^r_{p_s} w_{si}+ \hstar_{\ell} \left( - c_{ij,r}Y^{\ell}_{p_j} Y^r_{p_k}+ w_{si} Y^{\ell}_{p_kp_s} \right) $$
Therefore,
\begin{align*}
 f_k \b^k_i\tau^{i} & =\hstar_{\ell r} Y^{\ell}_{p_k} Y^r_{p_s} f_k w_{si} \tau^{i} - c_{ij,r}\hstar_{\ell}Y^{\ell}_{p_j} Y^r_{p_k} f_k \t^i +  \hstar_{\ell} Y^{\ell}_{p_kp_s} f_k w_{si} \t^i \\
& = \hstar_{\ell r} Y^{\ell}_{p_k} Y^r_{p_s} f_k f_s - c_{ij,r}\hstar_{\ell}Y^{\ell}_{p_j} Y^r_{p_k} f_k \t^i  + \hstar_{\ell} Y^{\ell}_{p_kp_s} f_k f_s.
\end{align*}
Substituting into \eqref{DbetaFgencost1} and observing that the second term in the above expression cancels the term $-c_{jk,r}\hstar_s  Y^r_{p_{\ell}} Y^s_{p_k} f_{\ell} \tau^{j}$ in \eqref{DbetaFgencost1}, we obtain
$$D_{\beta}F(x_0,t_0) =  t_0\left[-\chi \left(\nu^j_i  - c^{r,\ell} c_{ij,r}\nu^{\ell}\right) \tau^{i}\tau^{j} - \left(\hstar_{\ell r} Y^{\ell}_{p_k} Y^r_{p_s} +  \hstar_{\ell} Y^{\ell}_{p_kp_s}\right)f_k f_s + \a \beta^k_t f_k \right]\bigg|_{(x_0,t_0)}.$$
Next, we compute
$$\partial_t \beta^k  = \left(\hstar_{\ell r} Y^{\ell}_{p_k} Y^r_{p_s} + \hstar_{\ell} Y^{\ell}_{p_k p_s} \right) \theta_s.$$
Finally, noticing that $\hstar_{\ell r} Y^{\ell}_{p_k} Y^r_{p_s} + \hstar_{\ell} Y^{\ell}_{p_k p_s} = \G_{p_k p_s}$, we obtain the claimed expression \eqref{DbetaFgencost3}.
\end{proof}

\section{Proof of Exponential Convergence}\label{sec: exponential}
With Propositions \ref{DbetaFquadcost} and \ref{DbetaFexpression} in hand, we may now prove our main result.
\begin{proof}[Proof of Theorem \ref{thm: main}]
Consider the function 
$$\Theta(x,t) = \sup\limits_{\Omega} \theta(\cdot, 0)  - \theta(x, t)$$
which satisfies \eqref{linear}, and is non-negative by Proposition \ref{maxprinciple}. We claim $D_{\beta}F(x_0,t_0) \leq 0$ when $v = \Theta$, which will contradict Hopf's Lemma, thus proving $F$ cannot attain a positive maximum on $\partial \Omega\times (0, T)$.

Let us first deal with the case of the inner product cost. Since the domain $\O$ is convex, we have $\la (D \nu) \t, \t \ra \geq 0$. Therefore, since $\chi \geq 0$, we obtain using Proposition \ref{DbetaFquadcost}
\begin{equation}\label{DbetaF}
D_{\beta}F(x_0,t_0) \leq t_0\left[- \la D^2\hstar(\nabla u) \nabla f, \nabla f \ra + \a \la D^2 \hstar(\nabla u) \nabla \theta, \nabla f \ra \right]\bigg|_{(x_0,t_0)}.
\end{equation}
Next, the convexity of $\O^*$ implies $D^2\hstar$ is non-negative, so by substituting for $f = \log \Theta$ in \eqref{DbetaF}, we find
$$D_{\beta}F(x_0,t_0) \leq t_0\left[- \frac{1}{\Theta^2}\la D^2\hstar(\nabla u) \nabla \theta, \nabla \theta \ra - \frac{\a}{\Theta} \la D^2 \hstar(\nabla u) \nabla \theta, \nabla \theta \ra \right]\bigg|_{(x_0,t_0)} \leq 0.$$
This is the desired contradiction to Hopf's Lemma. For general costs, we use Proposition \ref{DbetaFexpression}, noticing that $c$-convexity of $\O$ with respect to $\O^*$ \eqref{cconvexity} implies $\left( \nu^j_i - c^{r,\ell} c_{ij,r}\nu^{\ell}\right) \tau^{i}\tau^{j} \geq 0$, while the $c^*$-convexity of $\O^*$ with respect to $\O$ \eqref{cstarconvexity} implies $\G_{p_k p_s}$ is a non-negative matrix.

It follows from Proposition \ref{prop: sublinearity if no boundary max} that with the choice $v = \Theta$, the corresponding function $F$ defined in \eqref{aux} is sublinear in time, and consequently the Harnack inequality \eqref{harnack} holds for $\Theta$. Using this Harnack inequality, we now prove exponential convergence of $\theta(\cdot, t)$. The argument is similar to \cite[Section 7]{KimStreetsWarren12}, but differs in an essential manner. For each integer $k \geq 1$, consider the functions
$$\Theta_k(x,t) := \sup\limits_{\Omega} \theta(\cdot, k-1)  - \theta(x, (k-1) + t).$$
The functions $\Theta_k$ are non-negative by Proposition \ref{maxprinciple} and solve \eqref{linear}. Arguing as above, the corresponding functions $F$ for $v = \Theta_k$ are also sublinear in $t$ (with constants independent of $k$) and thus the Harnack inequality \eqref{harnack} holds for $\Theta_k$. Applying \eqref{harnack} to $\Theta_k$ at $t = 1$ yields
\begin{equation}\label{upper}
\sup\limits_{\Omega} \theta(\cdot, k-1)  - \inf\limits_{\Omega}\theta (\cdot, k) \leq C\left( \sup\limits_{\Omega} \theta(\cdot, k-1)  - \sup\limits_{\Omega}  \theta\left(\cdot, k + 1\right)  \right)
\end{equation}
Now by \eqref{supandinf}, we know $\inf\limits_{\Omega}\theta\left(\cdot, k \right) \leq 0$ for each $k$. Therefore, defining $\e := \frac{C-1}{C} < 1$, we find
$$\sup\limits_{\Omega} \theta(\cdot, k+1) \leq \e \sup\limits_{\Omega} \theta(\cdot, k-1).$$
Iterating this inequality gives the exponential decay of the supremum
\begin{equation}\label{expconvofsup}
\sup\limits_{\Omega} \theta(\cdot, t) \leq \sup\limits_{\Omega} \theta(\cdot, 0)  e^{-\sigma t}, \text{ where } e^{-\sigma} = \e.
\end{equation}
On the other hand, \eqref{upper} implies
$$\inf\limits_{\Omega}\theta\left(\cdot, k \right)  \geq - (C-1)\sup\limits_{\Omega} \theta(\cdot, k-1) + C \sup\limits_{\Omega}  \theta(\cdot, k + 1) \geq - (C-1)\sup\limits_{\Omega} \theta(\cdot, k-1),$$
where we have used \eqref{supandinf} again to throw away the term $\sup\limits_{\Omega} \theta(\cdot, k + 1)$. Therefore, by \eqref{expconvofsup}, we obtain
\begin{equation}\label{expconvofinf}
\inf\limits_{\Omega}\theta\left(\cdot, k \right)  \geq - (C-1)\sup\limits_{\Omega} \theta(\cdot, 0) e^{-\sigma(k-1)}.
\end{equation}
This implies the exponential convergence of $\inf\limits_{\Omega} \theta(\cdot, t)$, which combined with \eqref{expconvofsup} gives the desired exponential convergence of $\theta(\cdot, t)$ to zero.
\end{proof}
\section{A Geometric Approach to Sublinearity at the Boundary}\label{sec: kim mccann}

In this section, we present an alternative approach to the computation of $D_\beta F(x_0, t_0)$ arising in the boundary sublinearity above. We will accomplish this using geometric language, exploiting the pseudo-Riemannian framework for optimal transport developed by Kim and McCann in \cite{KimMcCann10}.

In order to stay in line with established conventions, in this section we will mostly follow the notation used in \cite{KimMcCann10}. Thus in this section only, we will refer to the source and target domains as $\Omega$ and $\Omegabar$ respectively (in particular, $\Omegabar$ does not denote the closure of a set) which we assume are subsets of some fixed Riemannian manifolds. Points with a bar above will belong to $\Omegabar$, while those without will belong to $\Omega$. We also adopt the Einstein summation convention with the caveat that any indices given by greek letters will run from $1$ to $2n$, while lower case roman indices run between $1$ and $n$ with the convention that an index with a bar above will be that value with $n$ added to it: in otherwords, $1\leq \gamma\leq 2n$, $1\leq i \leq n$ and $\bar i:=i+n$. 

Additionally, we will switch sign conventions at this point to stay in line with the definitions of \cite{KimMcCann10}. This means that $c$ will be replaced by $-c$ everywhere, and the optimal transport problem \eqref{OT} that is considered will be a minimization instead of a maximization problem.

We also split the tangent and cotangent spaces of $\Omega\times \Omegabar$ in the canonical way according to the product structure, which gives the splitting $dc=Dc\oplus \Dbar c$ of the one form $dc$ on $\Omega\times \Omegabar$, and given any local coordinate system on $\Omega\times \Omegabar$ we will use the notation $X$ to denote the full $2n$ dimensional coordinate variable: thus given a point $X=(x, \xbar)\in \Omega\times \Omegabar$,  $X^i$ will indicate the $i$th coordinate of $x$ with $1\leq i\leq n$, and $X^{\bar i}$ will indicate the $i$th coordinate of $\bar x$. We will also suppress the time variable in this section, as everything considered will be for a fixed time $t$ (in fact, the time dependency of the potential $u$ will be completely irrelevant in the results of this section). Finally, we use the notation 
$$ [\Omega]_{\xbar}:=-\Dbar c(\Omega, \xbar)\subset T^*_{\xbar}\Omegabar, \qquad [\Omegabar]_{x}:=-Dc(x, \Omegabar)\subset T^*_x\Omega, \qquad \text{for any } (x, \xbar)\in\Omega\times\Omegabar.$$

Equip $\Omega$ with the pullback metric $w:=(\Id \times T)^*h$, where 
\begin{align*}
h:=\frac{1}{2}
\begin{pmatrix}
 0&-\Dbar Dc\\-D\Dbar c&0
\end{pmatrix}
\end{align*}
is the Kim-McCann (pseudo-Riemannian) metric on $\Omega\times \Omegabar$ defined as in \cite[(2.1)]{KimMcCann10}.
By \cite[Section 3.2]{KimMcCannWarren10}, in Euclidean coordinates the coefficients of $w$ at $x$ are exactly $w_{ij}(x)=u_{ij}(x)+c_{ij}(x, T(x))$, and $w$ is a Riemannian metric. We will write $\nabla^w$ and $\nabla^{h}$ for the Levi-Civita connections of $w$ and $h$ respectively, $\Gamma$ for the Christoffel symbols of $h$, and $\norm{\cdot}_w$ for the length of a vector in $w$. We will also metrically identify various cotangent spaces naturally with $\R^n$ through the underlying Riemannian metrics on $\Omega$ or $\Omegabar$. The inner products and norms in these underlying metrics will be denoted by $\la \cdot, \cdot \ra$ and $\norm{\cdot}$ respectively. Our main result of the section is the following.
\begin{Theorem}\label{thm: second fundamental forms}
Let $\II^w$ be the second fundamental form of $\partial \Omega$ defined with respect to the metric $w$, and fix a point $x_0\in \partial \Omega$. If $\II^{\partial\coord{\Omega}{T(x_0)}}$, $\II^{\partial\coord{\Omegabar}{x_0}}$ are the (Euclidean) second fundamental forms of $\partial\coord{\Omega}{T(x_0)}$ and $\partial\coord{\Omegabar}{x_0}$ respectively, then for any $\tau_1$, $\tau_2\in T_{x_0}\partial \Omega$ we have:
\begin{align}\label{eqn: II equality}
2\norm{\beta(x_0)}_w \II^w_{x_0}(\tau_1, \tau_2)=\norm{DT(x_0)\beta(x_0)}\II^{\partial\coord{\Omega}{T(x_0)}}_{-\Dbar c(x_0, T(x_0))}(\hat{\tau}_1, \hat{\tau}_2)+\norm{\beta(x_0)}\II^{\partial\coord{\Omegabar}{x_0}}_{-D c(x_0, T(x_0))}(\hat{\overline{\tau}}_1, \hat{\overline{\tau}}_2)
\end{align}
where 
\begin{align*}
\hat{\tau}_i:&=-D\Dbar c(x_0, T(x_0))\tau_i\in T^*_{T(x_0)}\Omegabar,\\
 \hat{\overline{\tau}}_i:&=-\Dbar D c(x_0, T(x_0))DT(x_0)\tau_i\in T^*_{x_0}\Omega.
\end{align*}
\end{Theorem}

\begin{proof}
Fix any point $x_0\in \partial \Omega$. Note by \eqref{Wbetaisnormal} that $\beta(x_0)$ is an (outward) normal to $\partial \Omega$ at $x_0$ with respect to the metric $w$. Then since $\Id\times T$ is an embedding of $\Omega$ into $\Omega\times \Omegabar$, if $\nabla^{h}$ is the Levi-Civita connection of $h$, we have (using that $\tau_2$ is tangent to $\partial \Omega$ in the second line)
\begin{align}
\II_{x_0}^w(\tau_1, \tau_2)&=w\left(\nabla^w_{\tau_1}\frac{\beta}{\norm{\beta}_w}, \tau_2\right)=\norm{\beta}_w^{-1}w(\nabla^w_{\tau_1}\beta, \tau_2)+D_{\tau_1}(\norm{\beta}_w^{-1})w(\beta, \tau_2)\notag\\
 &=\norm{\beta}_w^{-1}w(\nabla^w_{\tau_1}\beta, \tau_2)=-\norm{\beta}_w^{-1}w(\beta, \nabla^w_{\tau_1}\tau_2)\notag\\
 &=-\norm{\beta}_w^{-1}h\left((\beta \oplus DT(x_0) \beta), \nabla^{h}_{(\tau_1 \oplus DT(x_0) \tau_1)}(\tau_2 \oplus DT(x_0) \tau_2)\right)\notag\\
 &=-\norm{\beta}_w^{-1}(\beta \oplus DT(x_0) \beta)^\flat \left[\nabla^{h}_{(\tau_1 \oplus DT(x_0) \tau_1)}(\tau_2 \oplus DT(x_0) \tau_2)\right],\label{eqn: second FF in KM}
\end{align}
where $\flat$ is the operation of lowering the indices of a tangent vector to $\Omega\times \Omegabar$ by the metric $h$.  
Next consider the mapping $\Phi(x, \xbar):= -D c(x_0, \xbar)\oplus (-\Dbar c(x, T(x_0)))$. By the bi-twist condition \eqref{bi-twist}, $\Phi$ is a diffeomorphism on $\Omega\times \Omegabar$, hence $\Phi^{-1}$ gives a global coordinate chart on the set. We will use hats to denote quantities related to $h$ written in the coordinates given by $\Phi^{-1}$, while quantities without hats will be in Euclidean coordinates. A quick calculation yields that 
\begin{align}\label{eqn: metric equals coordinate change}
\frac{\partial\Phi^\delta}{\partial X^\gamma}(x_0, T(x_0))=2h_{\delta\gamma}{(x_0, T(x_0))}.
\end{align}

We will now calculate the Christoffel symbols $\hat \Gamma^\delta_{\gamma \lambda}$ in the coordinates given by $\Phi^{-1}$. By \cite[Lemma 4.1]{KimMcCann10} the Christoffel symbols of $h$ in Euclidean coordinates are identically zero unless all three of the indices are simultaneously between $1$ and $n$, or between $n+1$ and $2n$. Thus the standard transformation law shows that in the coordinates given by $\Phi^{-1}$, the only Christoffel symbols that can be nonzero are those where either the upper index is not barred and both lower indices are, or the upper index is barred and both lower indices are not. Since $\Omega$ is $c$-convex with respect to $\Omegabar$, there is an $n$-dimensional cone $K(x_0)$ of directions that point inward to $\coord{\Omega}{T(x_0)}$ from the boundary point $-\Dbar c(x_0, T(x_0))$. By \cite[Lemma 4.4]{KimMcCann10}, for any such direction $v$ in this cone $K(x_0)$, any segment of the form $s\mapsto \Phi^{-1}(sv\oplus -Dc(x_0, T(x_0)))$ is a geodesic for $h$, for small $s>0$. Thus plugging such a segment into the geodesic equations in $\Phi^{-1}$ coordinates yields for any fixed $\overline i$, at $(x_0, T(x_0))$,
\begin{align*}
 0=\hat\Gamma^{\overline i}_{jk}v^jv^k.
\end{align*}
Suppose $\{v_l\}_{l=1}^n$ is a linearly independent collection of vectors in $K(x_0)$, then for any $1\leq l_1\neq l_2\leq n$ we have
\begin{align*}
 0=\hat\Gamma^{\overline i}_{jk}(v_{l_1}^j+v_{l_2}^j)(v_{l_1}^k+v_{l_2}^k)=\hat\Gamma^{\overline i}_{jk}v_{l_1}^jv_{l_1}^k+\hat\Gamma^{\overline i}_{jk}v_{l_2}^jv_{l_2}^k+\hat\Gamma^{\overline i}_{jk}v_{l_1}^jv_{l_2}^k+\hat\Gamma^{\overline i}_{jk}v_{l_2}^jv_{l_1}^k=2\hat\Gamma^{\overline i}_{jk}v_{l_1}^jv_{l_2}^k,
\end{align*}
which implies all Christoffel symbols of the form $\hat\Gamma^{\overline i}_{jk}$ are also zero. A similar argument reversing the roles of $\Omega$ and $\Omegabar$ yields that \emph{all} Christoffel symbols of $h$ are zero in the $\Phi^{-1}$ coordinates at the point $(x_0, T(x_0))$.

Now using \eqref{eqn: metric equals coordinate change}, we see that the coordinates of the $1$-form $(\beta \oplus DT(x_0) \beta)^\flat$ in $\Phi^{-1}$ are equal to the Euclidean coordinates of the tangent vector $\frac{1}{2}(\beta \oplus DT(x_0) \beta)$. Also we can calculate for $i=1$ or $2$,
\begin{align*}
(\widehat{\tau_i\oplus DT(x_0)\tau_i})^j&=\frac{\partial\Phi^j}{\partial X^{\overline k}}(x_0, T(x_0))(\tau_i\oplus DT(x_0)\tau_i)^{\overline k}\\
&=-c_{j\overline k}(DT(x_0)\tau_i)^{k}=\hat{\overline{\tau}}_i^j,\\
(\widehat{\tau_i\oplus DT(x_0)\tau_i})^{\overline j}&=\frac{\partial\Phi^{\overline j}}{\partial X^{k}}(x_0, T(x_0))(\tau_i\oplus DT(x_0)\tau_i)^{k}\\
&=-c_{k\overline j}\tau_i^{k}=\hat{\tau}_i^j
\end{align*}
where we have identified $T^*_{x_0}\Omega$ and $T^*_{T(x_0)}\Omegabar$ with $\R^n$ to write the vectors $\hat\tau_i$ and $\hat{\overline{\tau}}_i$ defined in the statement of the theorem in Euclidean coordinates. Combining this fact with \eqref{eqn: metric equals coordinate change}, we can write \eqref{eqn: second FF in KM} in the coordinates given by $\Phi^{-1}$ as
\begin{align}\label{eqn: in good coordinates}
 -\frac{1}{2}\norm{\beta}_w^{-1}\left(\hat {\overline{\tau}}_1^j\sum_{i=1}^n\beta^i(\partial_{\hat x^j}\hat{\overline{\tau}}_2^i)+\hat{\tau}_1^l\sum_{k=1}^n(DT(x_0) \beta)^k(\partial_{\hat x^l}\hat{\tau}_2^k)\right).
\end{align}
Now we can see that the function $h^*(Y(x_0, \cdot))$ is a defining function for the set $\coord{\Omegabar}{x_0}$, hence identifying $T^*_{x_0}\Omega$ with $\R^n$ and differentiating yields that $\nabla_ph^*(Y(x_0, p))$ is in the outward normal direction for $p\in\partial \coord{\Omegabar}{x_0}$. In particular, the unit outward normal vector to $\partial \coord{\Omegabar}{x_0}$ at $-Dc(x_0, T(x_0))$ has coordinates given by $\dfrac{\beta^i}{\norm{\beta}}$. A similar calculation involving $h(X(T(x_0), \cdot))$ yields that the coordinates of the unit outward normal vector to $\partial \coord{\Omega}{T(x_0)}$ at $-\Dbar c(x_0, T(x_0))$ are given by $\dfrac{(DT(x_0) \beta)^k}{\norm{DT(x_0) \beta}}$. Additionally, since each $\tau_i$ is tangent to $\partial \Omega$, we see that $\hat \tau_i$ and $\hat{\overline{\tau}}_i$ are respectively tangent to $\partial \coord{\Omegabar}{x_0}$ and $\partial \coord{\Omega}{T(x_0)}$. Thus we calculate 
\begin{align*}
&\II_{-D c(x_0, T(x_0))}^{\partial\coord{\Omegabar}{x_0}}(\hat{\overline{\tau}}_1, \hat{\overline{\tau}}_2)=\la \nabla_{\hat{\overline{\tau}}_1}\frac{\beta}{\norm{\beta}}, \hat{\overline{\tau}}_2 \ra=\norm{\beta}^{-1}\inner{\nabla_{\hat{\overline{\tau}}_1}\beta}{\hat{\overline{\tau}}_2}+D_{\hat{\overline{\tau}}_1}\left(\frac{1}{\norm{\beta}}\right)\inner{\beta}{\hat{\overline{\tau}}_2}\\
&=\norm{\beta}^{-1}\inner{\nabla_{\hat{\overline{\tau}}_1}\beta}{\hat{\overline{\tau}}_2}=\norm{\beta}^{-1}(D_{\hat{\overline{\tau}}_1}\inner{\beta}{\hat{\overline{\tau}}_2}-\inner{\beta}{\nabla_{\hat{\overline{\tau}}_1}\hat{\overline{\tau}}_2})=-\norm{\beta}^{-1}\inner{\beta}{\nabla_{\hat{\overline{\tau}}_1}\hat{\overline{\tau}}_2}\\
&=-\norm{\beta}^{-1}\hat{\overline{\tau}}_1^j\sum_{i=1}^n\beta^i(\partial_{\hat x^j}\hat{\overline{\tau}}_2^i)\\
 \end{align*}
 and likewise
\begin{align*}
 &\II_{-\Dbar c(x_0, T(x_0))}^{\partial\coord{\Omega}{T(x_0)}}(\hat{\tau}_1, \hat{\tau}_2)=\la \nabla_{\hat{\tau}_1}\frac{DT(x_0) \beta}{\norm{DT(x_0) \beta}}, \hat{\tau}_2 \ra=-\norm{DT(x_0) \beta}^{-1}\inner{DT(x_0) \beta}{\nabla_{\hat{\tau}_1}\hat{\tau}_2}\\
 &=-\norm{DT(x_0) \beta}^{-1}\hat{\tau}_1^l\sum_{k=1}^n(DT(x_0) \beta)^k(\partial_{\hat x^l}\hat{\tau}_2^k).
\end{align*}
Comparing this with \eqref{eqn: in good coordinates} completes the proof of the theorem.
\end{proof}

The relevance of the above theorem to our current exponential convergence result is as follows. In terms of the metric $w$, we see that the $\beta$ directional derivative of the first term in the function $F$ defined by \eqref{aux} is given by (at $x_0$)
\begin{align*}
 D_\beta \left(w(\nabla^w f, \nabla^wf)\right)&=2w(\nabla^w_{\beta}\nabla^w f, \nabla^wf)=\Hess f(\beta, \nabla^wf)\\
 &=\Hess f( \nabla^wf, \beta)=2w(\nabla^w_{\nabla^w f}\nabla^wf, \beta)=-2w(\nabla^w_{\nabla^w f}\beta, \nabla^wf)\\
 &=-2\norm{\beta}_w\II^w(\nabla^w f, \nabla^w f).
\end{align*}
 Here we repeatedly used that $\nabla^wf$ is tangent to $\partial \Omega$ (due to the boundary condition $D_\beta v=0$ and since $f=\log v$) while $\beta$ is normal in the metric $w$, and we have used \eqref{eqn: second FF in KM} in the last line. Under the $c$- and $c^*$-convexity conditions \eqref{cconvexity} and \eqref{cstarconvexity}, the two terms on the right hand side of \eqref{eqn: II equality} are nonnegative, hence by Theorem \ref{thm: second fundamental forms}, $D_\beta w(\nabla^w f, \nabla^wf)$ is nonpositive. Thus in order to obtain a contradiction with the Hopf lemma as in Section \ref{sec: exponential}, all that remains is to evaluate the last term $-\alpha D_\beta(\partial_tf)$. Obtaining a sign on this term depends on the specific choice of the function $v$, as in Section \ref{sec: exponential}.
 
\section*{Acknowledgments}

The authors would like to thank Micah Warren for bringing our attention to the special case $n = 2$, and for pointing out the reference \cite{Warren14}.

\bibliography{parabolic}
\bibliographystyle{amsplain}

\end{document}